\documentclass[11pt]{article}
\usepackage[a4paper,margin=1in]{geometry}
\usepackage[T1]{fontenc}
\usepackage[utf8]{inputenc}
\usepackage{lmodern,amsmath,amssymb,amsthm,mathtools,bm,booktabs,array,microtype,enumitem}
\usepackage[hidelinks]{hyperref}
\usepackage[nameinlink,capitalize]{cleveref}
\newtheorem{theorem}{Theorem}[section]
\newtheorem{proposition}[theorem]{Proposition}
\newtheorem{lemma}[theorem]{Lemma}
\newtheorem{corollary}[theorem]{Corollary}
\theoremstyle{definition}
\theoremstyle{remark}\newtheorem{remark}[theorem]{Remark}
\newcommand{\R}{\mathbb R}\newcommand{\dd}{\,d}\newcommand{\HH}{\mathcal H}
\newcommand{\Qhat}{\widehat Q}\newcommand{\Qp}{\widehat{\mathfrak Q}}

\title{\textbf{Threshold Decomposition of Vorticity Stretching and Palinstrophy}}
\author{Rishad Shahmurov\\\small Cellular Products Research and Development, Roswell, GA, USA}
\date{September 2026}
\begin{document}\maketitle
\begin{abstract}
We give an exact finite-time decomposition of vortex stretching by resolving the vorticity equation according to the size of the vorticity.  For almost every threshold $a>0$, the stretching occurring where $|\omega|>a$ equals the change of the corresponding superlevel mass plus two nonnegative viscous terms: one measuring changes in vorticity direction and one measuring dissipation across the level surface $|\omega|=a$.  Integrating this identity over all thresholds recovers the classical $L^p$ vorticity balance and, for $p=2$, the usual spacetime palinstrophy.  At the scale-critical exponent $p=3/2$, the dissipation can be written through the nonlinear variable $Z=|\omega|^{-1/4}\omega$.  The threshold formulation also gives a positive measure that can be used to assign overlapping high-vorticity episodes without double counting.  Under additional compactness assumptions it leads, through empirical transition couplings supported on a closed physical transition relation, to a stationary Markov renormalization law; no deterministic return map is required.  These are structural results; no unconditional three-dimensional regularity theorem is claimed.
\end{abstract}

\section{Introduction}
The vorticity equation is one of the clearest places where the difficulty of three-dimensional Navier--Stokes flow appears.  If $\omega=\nabla\times u$, then
\[
\partial_t\omega+u\cdot\nabla\omega
=\omega\cdot\nabla u+\nu\Delta\omega.
\]
The viscous term smooths the vorticity, whereas the stretching term $\omega\cdot\nabla u$ can increase its magnitude.  In two dimensions the stretching mechanism is absent; in three dimensions it is the central nonlinear effect.  Leray's weak-solution theory and the partial-regularity theory of Caffarelli, Kohn, and Nirenberg provide the classical global and local frameworks in which possible singular behavior is studied \cite{Leray1934,CKN1982}.  Geometric criteria such as the Constantin--Fefferman condition show that the direction of vorticity can also play a decisive regularizing role \cite{ConstantinFefferman1993}.

Most integral estimates average over the full range of vorticity amplitudes.  For example, multiplying the vorticity-magnitude equation by a power of $|\omega|$ leads to the familiar $L^p$ balances.  Such identities are fundamental, but they no longer remember at which amplitude level a given part of the stretching or dissipation occurred.  The starting point of this paper is to retain that amplitude level as a variable.  Instead of integrating immediately over all values of $|\omega|$, we first ask what the vorticity equation says on the superlevel region
\[
\{x:|\omega(x,t)|>a\}
\]
for one fixed threshold $a$.

This leads to a simple but useful exact identity.  Write
\[
f=|\omega|,\qquad \xi=\frac{\omega}{|\omega|},\qquad
\alpha=\xi\cdot S[u]\xi,
\]
where $S[u]$ is the strain tensor.  For almost every $a>0$, the total stretching above the level $a$ on a finite time interval can be written as three pieces: the change between the initial and final superlevel masses, a bulk dissipation involving $f|\nabla\xi|^2$, and a level-surface dissipation involving $|\nabla f|$ on $\{f=a\}$.  The last two terms are nonnegative.  The endpoint increment is essential.  Dropping it would turn a finite-interval identity into a false monotonicity statement.  Keeping it is what allows the threshold decomposition to remain an exact reformulation of the smooth Navier--Stokes equations.

The identity has two complementary interpretations.  Analytically, it refines the usual vorticity balance.  Integrating in $a$ with the weight $a^{p-2}$ reconstructs the classical $L^p$ dissipation.  At $p=2$, the two threshold dissipations recombine exactly into
\[
\nu\int|\nabla\omega|^2,
\]
the spacetime palinstrophy.  Thus the threshold description does not introduce a new energy law; it resolves a familiar one into contributions from different amplitude levels.  Geometrically, the two nonnegative pieces separate changes in the direction of vorticity from changes in its scalar magnitude.  This makes precise, at the level of an identity, the distinction between directional coherence and amplitude variation that underlies many geometric regularity arguments.

The scale-critical exponent $p=3/2$ is especially natural.  At this exponent the vorticity norm is invariant under the Navier--Stokes scaling.  The corresponding dissipation can be encoded by the nonlinear root
\[
Z=|\omega|^{-1/4}\omega.
\]
The gradient of $Z$ contains exactly the two polar pieces already present in the threshold identity: one coming from $\nabla f$ and one from $\nabla\xi$.  This gives a compact way to describe critical dissipation without pretending that the two mechanisms are the same.  Later sections use this representation to distinguish twisting of the vorticity direction from a transverse amplitude sheath in normalized critical packets.

A second purpose of keeping the threshold variable is measure-theoretic.  Because the dissipative part of the threshold identity is nonnegative for almost every $a$, it defines a positive Radon measure on threshold--time space.  This is useful when many candidate high-vorticity episodes overlap.  If one simply selects cylinders independently, the same physical dissipation may be counted many times.  The positive threshold measure permits a canonical first assignment: a piece of positive mass is charged when it is first encountered and is not charged again merely because a later selected event overlaps it.  The construction is bookkeeping in a precise measure-theoretic sense; it is not an additional physical conservation law.

The paper also records the localized version of the identity.  When a spatial cutoff moves or changes scale with the selected event, the transport of the cutoff produces explicit center-motion, scale-motion, and diffusion terms.  These terms are important in blow-up arguments because a moving window can otherwise hide flux through its boundary.  The formula here keeps those contributions visible.  This viewpoint is compatible with the suitable-weak and local-pressure methods used in endpoint analysis; in particular, Type-I blow-up and Liouville arguments emphasize how much information must survive a rescaling procedure \cite{AlbrittonBarker2019}, while local pressure expansions clarify how nonlocal pressure should be retained when one passes from a global solution to a local endpoint description \cite{BradshawTsai2022}.  The present paper does not prove those endpoint theorems, but its threshold measure is designed to interface with such compactness arguments.

The final structural step is a renormalization statement under compactness assumptions.  Suppose a sequence of normalized threshold events remains in a compact decorated state chamber and consecutive states belong to one closed physical transition relation.  Empirical measures of the consecutive pairs then have subsequential limits whose two state marginals agree.  Disintegration gives a stationary Markov law supported on the exact transition relation.  No continuous single-valued return map is needed.  The threshold dissipation measure and the first-assignment rule travel with the normalized state, while the limiting law records only those transitions that survive the closed-relation passage.  The conclusion should not be confused with deterministic self-similarity or with the existence of an atomic recurrent profile: a stationary law may be atomless and need not repeat any single normalized state.

The main contribution of the paper is therefore organizational but exact.  The classical vorticity-magnitude equation, coarea formula, $L^p$ balances, palinstrophy identity, and geometric role of the vorticity direction are not presented as new.  What is new in the present framework is to keep the amplitude threshold and the finite-time endpoint term throughout the analysis, to package the resulting positive dissipation as a measure, and to use that measure to formulate nonoverlapping assignment and compact renormalization.  This separates several issues that are easily conflated: stretching versus endpoint change, direction versus amplitude dissipation, physical dissipation versus selection bookkeeping, and compact statistical recurrence versus self-similar dynamics.

The scope is deliberately limited.  All identities are proved for smooth solutions, with localized versions suitable for later limiting arguments.  The compactification statement requires the explicit compactness hypotheses stated in its section.  None of these results alone excludes a Navier--Stokes singularity.  Their purpose is to provide a clean threshold-resolved language in which a later endpoint or recurrence argument can be formulated without losing the nonnegative dissipation already present in the classical equation.

The paper is organized as follows.  Section 2 derives the vorticity-magnitude equation and the endpoint-corrected threshold identity, including its moving-cutoff form.  Section 3 reconstructs the classical $L^p$ balance and introduces the critical nonlinear root.  Section 4 separates directional twisting from transverse amplitude variation.  Section 5 constructs the threshold-dissipation measure and the first-assignment rule.  Section 6 gives the renormalization covariance and stationary compactification, and the final sections record the limitations and the relation with the classical vorticity balance.

\section{Vorticity magnitude and the finite-interval threshold identity}
Consider a smooth solution on $\R^3\times I$, $I=[t_0,t_1]$, with sufficient decay for the integrations below:
\[
\partial_tu+u\cdot\nabla u+\nabla p=\nu\Delta u,\qquad \nabla\cdot u=0.
\]
Write
\[
\omega=\nabla\times u,\qquad f=|\omega|,\qquad \xi=\omega/f\quad(f>0),\qquad \alpha=\xi\cdot S[u]\xi,
\]
where
\[
S[u]=\frac12\bigl(\nabla u+(\nabla u)^T\bigr)
\]
is the symmetric strain tensor.

\medskip
\noindent\textbf{Vorticity-magnitude identity.}
\begin{lemma}\label{lem:magnitude}
On $\{f>0\}$,
\[
(\partial_t+u\cdot\nabla-\nu\Delta)f
=\alpha f-\nu f|\nabla\xi|^2.
\]
\end{lemma}
\begin{proof}
We give the computation in detail because the two dissipative terms that appear
later are already visible here.

Take the curl of the Navier--Stokes equation.  Since
$\nabla\times\nabla p=0$ and $\nabla\cdot u=0$, the standard vector identity
\[
\nabla\times\bigl((u\cdot\nabla)u\bigr)
=(u\cdot\nabla)\omega-(\omega\cdot\nabla)u
\]
gives
\[
\partial_t\omega+(u\cdot\nabla)\omega
=(\omega\cdot\nabla)u+\nu\Delta\omega.
\]
Write $\nabla u=S[u]+A[u]$, where $S[u]$ is symmetric and $A[u]$ is
antisymmetric.  In three dimensions
\[
A[u]v=\frac12\,\omega\times v,
\]
so in particular $A[u]\omega=0$.  Hence
\[
(\omega\cdot\nabla)u=(\nabla u)\omega=S[u]\omega,
\]
and therefore
\begin{equation}\label{eq:vorticity-vector}
(\partial_t+u\cdot\nabla-\nu\Delta)\omega=S[u]\omega.
\end{equation}

Now fix a point at which $f=|\omega|>0$.  There $f$ and
$\xi=\omega/f$ are smooth.  Differentiating $f^2=\omega\cdot\omega$ gives
\[
\partial_t f=\xi\cdot\partial_t\omega,
\qquad
\partial_jf=\xi\cdot\partial_j\omega,
\qquad
u\cdot\nabla f=\xi\cdot (u\cdot\nabla\omega).
\]
Thus the only term that requires care after taking the scalar product of
\eqref{eq:vorticity-vector} with $\xi$ is the Laplacian.

Because $\omega=f\xi$,
\[
\Delta(f\xi)
=(\Delta f)\xi+2\sum_{j=1}^3(\partial_jf)(\partial_j\xi)
+f\Delta\xi.
\]
The identity $|\xi|^2=1$ implies, for every $j$,
\[
\xi\cdot\partial_j\xi=0.
\]
Applying $\Delta$ to $|\xi|^2=1$ also gives
\[
0=\frac12\Delta|\xi|^2
=|\nabla\xi|^2+\xi\cdot\Delta\xi,
\]
so
\[
\xi\cdot\Delta\xi=-|\nabla\xi|^2.
\]
Consequently
\[
\xi\cdot\Delta(f\xi)
=\Delta f-f|\nabla\xi|^2.
\]
Finally,
\[
\xi\cdot S[u]\omega
=f\,\xi\cdot S[u]\xi
=\alpha f.
\]
Taking the scalar product of \eqref{eq:vorticity-vector} with $\xi$ and
substituting the preceding identities yields
\[
(\partial_t+u\cdot\nabla-\nu\Delta)f
=\alpha f-\nu f|\nabla\xi|^2,
\]
as claimed.
\end{proof}

\medskip
\noindent\textbf{Endpoint-corrected threshold identity.}
\begin{theorem}\label{thm:threshold}
For almost every $a>0$, define
\[
\Qhat_I(a)=\int_I\int_{\{f>a\}}\alpha f\,\dd x\,\dd t
-\left[\int_{\R^3}(f-a)_+\,\dd x\right]_{t_0}^{t_1}.
\]
Then
\[
\boxed{
\Qhat_I(a)=
\nu\int_I\int_{\{f>a\}}f|\nabla\xi|^2\,\dd x\,\dd t
+\nu\int_I\int_{\{f=a\}}|\nabla f|\,\dd\HH^2\,\dd t\ge0.
}
\]
\end{theorem}
\begin{proof}
Fix $a>0$.  We first derive a smooth identity and only then pass to the
nonsmooth truncation.  This avoids any ambiguity at points where $\omega=0$.

Choose a nonnegative even mollifier $\rho\in C_c^\infty((-1,1))$ with
$\int_\R\rho=1$, and set
\[
\rho_\varepsilon(s)=\varepsilon^{-1}\rho(s/\varepsilon).
\]
Let $H_\varepsilon$ be the smooth increasing function satisfying
\[
H_\varepsilon'=\rho_\varepsilon,
\qquad
0\le H_\varepsilon\le1,
\qquad
H_\varepsilon(s)=0\ \text{for }s\le-\varepsilon,
\qquad
H_\varepsilon(s)=1\ \text{for }s\ge\varepsilon.
\]
Define
\[
\Phi_{a,\varepsilon}(s)=\int_0^sH_\varepsilon(\sigma-a)\,\dd\sigma.
\]
Then $\Phi_{a,\varepsilon}$ is convex,
\[
\Phi_{a,\varepsilon}'(s)=H_\varepsilon(s-a),
\qquad
\Phi_{a,\varepsilon}''(s)=\rho_\varepsilon(s-a)\ge0,
\]
and $\Phi_{a,\varepsilon}(s)\to(s-a)_+$ locally uniformly as
$\varepsilon\downarrow0$.  We may take $0<\varepsilon<a/2$.  On the support
of $\Phi'_{a,\varepsilon}(f)$ one then has $f>a/2>0$, so
\cref{lem:magnitude} is valid pointwise wherever the multiplier is nonzero.

Multiply that identity by $\Phi'_{a,\varepsilon}(f)$ and integrate over
$\R^3$.  The chain rule gives
\[
\Phi'_{a,\varepsilon}(f)(\partial_t+u\cdot\nabla)f
=(\partial_t+u\cdot\nabla)\Phi_{a,\varepsilon}(f).
\]
Because $\nabla\cdot u=0$ and the solution has the stated decay,
\[
\int_{\R^3}u\cdot\nabla\Phi_{a,\varepsilon}(f)\,\dd x
=\int_{\R^3}\nabla\cdot\bigl(u\Phi_{a,\varepsilon}(f)\bigr)\,\dd x=0.
\]
For the Laplacian term, integration by parts and the chain rule give
\[
\begin{aligned}
\int_{\R^3}\Phi'_{a,\varepsilon}(f)\Delta f\,\dd x
&=-\int_{\R^3}\nabla\bigl(\Phi'_{a,\varepsilon}(f)\bigr)\cdot\nabla f\,\dd x\\
&=-\int_{\R^3}\Phi''_{a,\varepsilon}(f)|\nabla f|^2\,\dd x.
\end{aligned}
\]
We therefore obtain, for every such $\varepsilon$,
\begin{equation}\label{eq:smooth-threshold}
\begin{aligned}
\frac{\dd}{\dd t}\int_{\R^3}\Phi_{a,\varepsilon}(f)\,\dd x
={}&\int_{\R^3}\Phi'_{a,\varepsilon}(f)\alpha f\,\dd x\\
&-\nu\int_{\R^3}\Phi'_{a,\varepsilon}(f)f|\nabla\xi|^2\,\dd x\\
&-\nu\int_{\R^3}\rho_\varepsilon(f-a)|\nabla f|^2\,\dd x.
\end{aligned}
\end{equation}
Integrating from $t_0$ to $t_1$ produces the finite-interval identity with
smooth truncation.

We now pass to the limit.  First,
\[
\Phi_{a,\varepsilon}(f)\longrightarrow(f-a)_+,
\qquad
\Phi'_{a,\varepsilon}(f)\longrightarrow\mathbf1_{\{f>a\}}
\]
at every point with $f\ne a$.  For almost every $a$, the spacetime level set
$\{(x,t):f(x,t)=a\}$ has four-dimensional Lebesgue measure zero.  Indeed,
this follows from Fubini together with the elementary fact that a real-valued
measurable function has positive-measure level sets for at most countably
many levels on each finite-measure truncation of spacetime.  The smoothness
and decay assumptions give an integrable majorant for the stretching and
bulk-direction terms on the finite interval, so dominated convergence yields
\[
\int_I\!\int\Phi'_{a,\varepsilon}(f)\alpha f
\longrightarrow
\int_I\!\int_{\{f>a\}}\alpha f
\]
and
\[
\int_I\!\int\Phi'_{a,\varepsilon}(f)f|\nabla\xi|^2
\longrightarrow
\int_I\!\int_{\{f>a\}}f|\nabla\xi|^2.
\]
The endpoint term converges to
\[
\left[\int_{\R^3}(f-a)_+\,\dd x\right]_{t_0}^{t_1}.
\]

It remains to identify the last term in \eqref{eq:smooth-threshold}.  Define
for almost every $s>0$
\[
G(s)=\int_I\int_{\{f(\cdot,t)=s\}}|\nabla f|\,\dd\HH^2\,\dd t.
\]
The coarea formula, first in space and then integrated in time, gives for
any compact interval $[c,d]\subset(0,\infty)$
\[
\int_c^dG(s)\,\dd s
=\int_I\int_{\{c<f<d\}}|\nabla f|^2\,\dd x\,\dd t<\infty.
\]
Hence $G\in L^1_{\rm loc}(0,\infty)$.  A second application of coarea gives
\[
\begin{aligned}
\int_I\int_{\R^3}\rho_\varepsilon(f-a)|\nabla f|^2\,\dd x\,\dd t
&=\int_0^\infty\rho_\varepsilon(s-a)G(s)\,\dd s.
\end{aligned}
\]
The right-hand side is an approximate identity applied to $G$.  By the
Lebesgue differentiation theorem it converges to $G(a)$ for almost every
$a>0$.  Thus
\[
\lim_{\varepsilon\downarrow0}
\int_I\int\rho_\varepsilon(f-a)|\nabla f|^2
=
\int_I\int_{\{f=a\}}|\nabla f|\,\dd\HH^2\,\dd t
\]
for almost every $a$.

Passing to the limit in the time-integrated form of
\eqref{eq:smooth-threshold} and moving the endpoint increment to the left gives
exactly
\[
\Qhat_I(a)=
\nu\int_I\int_{\{f>a\}}f|\nabla\xi|^2
+\nu\int_I\int_{\{f=a\}}|\nabla f|\,\dd\HH^2\,\dd t.
\]
Both terms on the right are nonnegative, proving the final inequality as
well.
\end{proof}

\subsection{Localized and moving-cutoff identity}
The finite-interval identity has a useful exact localized form.  Let
$\phi\in C_c^\infty(I\times\R^3)$ be nonnegative and put $h_a=(f-a)_+$.
In distributions,
\[
(\partial_t+u\cdot\nabla-\nu\Delta)h_a
=
\mathbf 1_{\{f>a\}}\bigl(\alpha f-\nu f|\nabla\xi|^2\bigr)
-\nu\delta(f-a)|\nabla f|^2.
\]
Define
\[
\begin{aligned}
\Qhat_{I,\phi}(a):={}&
\int_I\int_{\{f>a\}}\phi\,\alpha f\,\dd x\dd t
-\left[\int_{\R^3}\phi h_a\,\dd x\right]_{t_0}^{t_1}\\
&+\int_I\int_{\R^3}h_a
\bigl(\partial_t\phi+u\cdot\nabla\phi+\nu\Delta\phi\bigr)\,\dd x\dd t.
\end{aligned}
\]
\medskip
\noindent\textbf{Localized threshold identity.}
\begin{theorem}\label{thm:localized-threshold}
For almost every regular level $a>0$,
\[
\boxed{
\begin{aligned}
\Qhat_{I,\phi}(a)
={}&\nu\int_I\int_{\{f>a\}}\phi f|\nabla\xi|^2\,\dd x\dd t\\
&+\nu\int_I\int_{\{f=a\}}\phi|\nabla f|\,\dd\HH^2\dd t
\ge0.
\end{aligned}}
\]
\end{theorem}
\begin{proof}
We first justify the distributional equation for $h_a=(f-a)_+$ rather than
using it formally.  Use the same convex approximation
$\Phi_{a,\varepsilon}$ as in the proof of \cref{thm:threshold}.  The chain
rule applied to \cref{lem:magnitude} gives
\[
\begin{aligned}
&(\partial_t+u\cdot\nabla-\nu\Delta)\Phi_{a,\varepsilon}(f)\\
&\qquad=
\Phi'_{a,\varepsilon}(f)
\bigl(\alpha f-\nu f|\nabla\xi|^2\bigr)
-\nu\Phi''_{a,\varepsilon}(f)|\nabla f|^2.
\end{aligned}
\]
Testing against an arbitrary compactly supported smooth function and letting
$\varepsilon\downarrow0$ exactly as in \cref{thm:threshold} yields, for
almost every $a>0$,
\[
(\partial_t+u\cdot\nabla-\nu\Delta)h_a
=
\mathbf1_{\{f>a\}}
\bigl(\alpha f-\nu f|\nabla\xi|^2\bigr)
-\nu\delta(f-a)|\nabla f|^2
\]
in distributions.  The notation involving $\delta(f-a)$ is shorthand for
the coarea measure characterized by
\[
\int_{\R^3}\psi(x)\delta(f-a)|\nabla f|^2\,\dd x
=
\int_{\{f=a\}}\psi|\nabla f|\,\dd\HH^2
\]
for every compactly supported continuous $\psi$, at almost every regular
level $a$.

Now multiply the distributional identity by the nonnegative cutoff $\phi$
and integrate over $I\times\R^3$.  We treat the three left-hand terms one at
a time.  Integration by parts in time gives
\[
\int_I\int\phi\,\partial_t h_a
=
\left[\int_{\R^3}\phi h_a\,\dd x\right]_{t_0}^{t_1}
-\int_I\int h_a\,\partial_t\phi.
\]
For transport, incompressibility gives
\[
\begin{aligned}
\int_I\int\phi\,u\cdot\nabla h_a
&=\int_I\int u\cdot\nabla(\phi h_a)
-\int_I\int h_a\,u\cdot\nabla\phi\\
&=-\int_I\int h_a\,u\cdot\nabla\phi,
\end{aligned}
\]
because the divergence integral vanishes.  For diffusion, two integrations
by parts, or equivalently self-adjointness of $\Delta$, give
\[
-\nu\int_I\int\phi\Delta h_a
=-\nu\int_I\int h_a\Delta\phi.
\]
Substituting these three expressions and moving all cutoff terms to the left
produces
\[
\begin{aligned}
&\int_I\int_{\{f>a\}}\phi\alpha f
-\left[\int\phi h_a\right]_{t_0}^{t_1}
+\int_I\int h_a
(\partial_t\phi+u\cdot\nabla\phi+\nu\Delta\phi)\\
&\qquad=
\nu\int_I\int_{\{f>a\}}\phi f|\nabla\xi|^2
+\nu\int_I\int\phi\delta(f-a)|\nabla f|^2.
\end{aligned}
\]
The expression on the first two lines is precisely
$\Qhat_{I,\phi}(a)$.  Using the coarea interpretation of the final term gives
\[
\int\phi\delta(f-a)|\nabla f|^2\,\dd x
=
\int_{\{f=a\}}\phi|\nabla f|\,\dd\HH^2.
\]
This proves the displayed identity.  Nonnegativity follows from
$\phi\ge0$ and from the nonnegativity of both dissipative integrands.
\end{proof}

For an actual moving window let
\[
\phi(t,x)=\eta\!\left(\frac{x-X(t)}{R(t)}\right),
\qquad
y=\frac{x-X(t)}{R(t)},
\]
where $X:I\to\R^3$ and $R:I\to(0,\infty)$ are $C^1$.  The cutoff terms can
be computed without approximation.  Since
\[
\partial_t y
=-\frac{\dot X}{R}-\frac{\dot R}{R}y,
\qquad
\nabla_x y=R^{-1}I,
\]
the chain rule gives
\[
\partial_t\phi
=-\frac{\dot X}{R}\cdot\nabla\eta
-\frac{\dot R}{R}y\cdot\nabla\eta,
\qquad
u\cdot\nabla\phi
=\frac{u}{R}\cdot\nabla\eta.
\]
Adding these formulas yields
\[
(\partial_t+u\cdot\nabla)\phi
=
\frac{u-\dot X}{R}\cdot\nabla\eta
-\frac{\dot R}{R}\,y\cdot\nabla\eta.
\]
Every spatial derivative contributes one factor $R^{-1}$, so
\[
\Delta_x\phi=R^{-2}\Delta_y\eta.
\]
Thus the moving-center flux, moving-scale flux and diffusion-cutoff flux are
not informal errors: they are the three explicit terms generated by the
localized identity.  None may be silently absorbed into the positive
threshold measure.

\section{Layer cake, the nonlinear vorticity root, and critical scaling}
\medskip
\noindent\textbf{Weighted threshold reconstruction.}
\begin{theorem}\label{thm:layercake}
For $p>1$,
\[
\Qp_{p,I}:=\int_0^\infty a^{p-2}\Qhat_I(a)\,\dd a
\]
satisfies
\[
\boxed{
\Qp_{p,I}=
\frac{\nu}{p-1}\iint_I f^p|\nabla\xi|^2
+\nu\iint_I f^{p-2}|\nabla f|^2.
}
\]
\end{theorem}
\begin{proof}
By \cref{thm:threshold}, $\Qhat_I(a)\ge0$ for almost every $a$.  Therefore
the weighted integral defining $\Qp_{p,I}$ is an integral of a nonnegative
measurable function, and Tonelli's theorem applies even before finiteness is
known.  We may consequently integrate the two nonnegative terms in
\cref{thm:threshold} separately.

For the directional term, fix $(x,t)$ and abbreviate $F=f(x,t)$.  The point
$(x,t)$ belongs to the superlevel set $\{f>a\}$ exactly for $0<a<F$.  Hence
\[
\int_0^\infty a^{p-2}\mathbf1_{\{a<f(x,t)\}}\,\dd a
=\int_0^F a^{p-2}\,\dd a
=\frac{F^{p-1}}{p-1},
\]
where $p>1$ is precisely what makes the integral finite at $a=0$ for fixed
$F$.  Multiplying by $f|\nabla\xi|^2$ and applying Tonelli gives
\[
\begin{aligned}
&\int_0^\infty a^{p-2}
\left(\int_I\int_{\{f>a\}}f|\nabla\xi|^2\,\dd x\,\dd t\right)\dd a\\
&\qquad=\frac1{p-1}\int_I\int_{\R^3}f^p|\nabla\xi|^2\,\dd x\,\dd t.
\end{aligned}
\]

For the level-surface term, apply the coarea formula at each time to the
nonnegative function
\[
g(x)=f(x,t)^{p-2}|\nabla f(x,t)|.
\]
In the form needed here, coarea says
\[
\int_{\R^3}g(x)|\nabla f(x,t)|\,\dd x
=
\int_0^\infty\left(
\int_{\{f=a\}}g\,\dd\HH^2\right)\dd a.
\]
On the level set $\{f=a\}$, one has $f^{p-2}=a^{p-2}$.  Thus
\[
\int_0^\infty a^{p-2}
\int_{\{f=a\}}|\nabla f|\,\dd\HH^2\,\dd a
=
\int_{\R^3}f^{p-2}|\nabla f|^2\,\dd x.
\]
Integrating this equality in time and inserting the factor $\nu$ proves
\[
\Qp_{p,I}=
\frac{\nu}{p-1}\iint_I f^p|\nabla\xi|^2
+\nu\iint_I f^{p-2}|\nabla f|^2.
\]
All equalities are valid in $[0,\infty]$; whenever one side is finite, both
terms on the right are finite as well.
\end{proof}

\medskip
\noindent\textbf{Recovery of the classical finite-time $L^p$ balance.}
\begin{corollary}\label{cor:classical-lp}
For $p>1$, whenever the displayed quantities are finite,
\[
\boxed{
\begin{aligned}
\frac1p\left[\int_{\R^3}f^p\,\dd x\right]_{t_0}^{t_1}
&+\nu\iint_I f^p|\nabla\xi|^2\,\dd x\,\dd t\\
&+\nu(p-1)\iint_I f^{p-2}|\nabla f|^2\,\dd x\,\dd t
=\iint_I\alpha f^p\,\dd x\,\dd t.
\end{aligned}}
\]
\end{corollary}
\begin{proof}
We compute the left-hand definition of $\Qp_{p,I}$, rather than its
dissipative representation.  For the stretching part, Tonelli gives
\[
\begin{aligned}
\int_0^\infty a^{p-2}
\int_I\int_{\{f>a\}}\alpha f\,\dd x\,\dd t\,\dd a
&=\int_I\int\alpha f
\left(\int_0^f a^{p-2}\,\dd a\right)\dd x\,\dd t\\
&=\frac1{p-1}\iint_I\alpha f^p.
\end{aligned}
\]
For the endpoint term, the required one-dimensional integral is
\[
\begin{aligned}
\int_0^\infty a^{p-2}(F-a)_+\,\dd a
&=\int_0^F a^{p-2}(F-a)\,\dd a\\
&=F\frac{F^{p-1}}{p-1}-\frac{F^p}{p}
=\frac{F^p}{p(p-1)}.
\end{aligned}
\]
Consequently
\[
\Qp_{p,I}
=\frac1{p-1}\iint_I\alpha f^p
-\frac1{p(p-1)}
\left[\int f^p\right]_{t_0}^{t_1}.
\]
Equate this formula with \cref{thm:layercake} and multiply by $p-1$.
Rearranging gives exactly the stated $L^p$ balance.
\end{proof}

\medskip
\noindent\textbf{Palinstrophy recombination.}
\begin{corollary}\label{cor:pal}
At $p=2$,
\[
\boxed{\Qp_{2,I}=\nu\iint_I|\nabla\omega|^2.}
\]
\end{corollary}
\begin{proof}
Differentiate $\omega=f\xi$ componentwise.  For each spatial direction $j$,
\[
\partial_j\omega=(\partial_jf)\xi+f\partial_j\xi.
\]
Since $|\xi|^2=1$, differentiation gives
$\xi\cdot\partial_j\xi=0$.  The two vectors on the right are therefore
orthogonal, and
\[
|\partial_j\omega|^2
=|\partial_jf|^2+f^2|\partial_j\xi|^2.
\]
Summing over $j=1,2,3$ yields
\[
|\nabla\omega|^2=|\nabla f|^2+f^2|\nabla\xi|^2
\]
on $\{f>0\}$.  The identity extends in the integrated sense across the zero
set by approximation, or directly from the smooth vector field $\omega$.
Setting $p=2$ in \cref{thm:layercake} gives
\[
\Qp_{2,I}
=\nu\iint_I f^2|\nabla\xi|^2
+\nu\iint_I|\nabla f|^2
=\nu\iint_I|\nabla\omega|^2,
\]
which is the desired recombination.
\end{proof}

\medskip
\noindent\textbf{Nonlinear vorticity root.}
\begin{proposition}\label{prop:root}
For $p>1$ define
\[
Z_p=f^{p/2}\xi=f^{p/2-1}\omega.
\]
Then pointwise on $\{f>0\}$,
\[
\boxed{|\nabla Z_p|^2=\frac{p^2}{4}f^{p-2}|\nabla f|^2+f^p|\nabla\xi|^2.}
\]
At $p=3/2$, with $Z=f^{-1/4}\omega$,
\[
\boxed{
\frac{16\nu}{9}\iint|\nabla Z|^2
\le \Qp_{3/2,I}
\le 2\nu\iint|\nabla Z|^2.
}
\]
\end{proposition}
\begin{proof}
Fix $p>1$ and work first on $\{f>0\}$.  For every spatial index $j$,
\[
\partial_jZ_p
=\frac p2 f^{p/2-1}(\partial_jf)\xi
+f^{p/2}\partial_j\xi.
\]
Again $\xi\cdot\partial_j\xi=0$, so the two summands are orthogonal.  Hence
\[
|\partial_jZ_p|^2
=\frac{p^2}{4}f^{p-2}|\partial_jf|^2
+f^p|\partial_j\xi|^2.
\]
Summing over $j$ proves
\[
|\nabla Z_p|^2
=\frac{p^2}{4}f^{p-2}|\nabla f|^2
+f^p|\nabla\xi|^2.
\]
This formula also gives the natural weak definition of the gradient across the
zero set whenever the right-hand side is integrable.

Now set $p=3/2$ and write
\[
A=\iint_I f^{3/2}|\nabla\xi|^2,
\qquad
B=\iint_I f^{-1/2}|\nabla f|^2.
\]
Then \cref{thm:layercake} gives
\[
\nu^{-1}\Qp_{3/2,I}=2A+B,
\]
whereas the pointwise identity just proved gives
\[
\iint_I|\nabla Z|^2=A+\frac9{16}B.
\]
For the lower bound, multiply the latter by $16/9$:
\[
\frac{16}{9}\iint|\nabla Z|^2
=\frac{16}{9}A+B
\le2A+B
=\nu^{-1}\Qp_{3/2,I},
\]
because $16/9\le2$.  For the upper bound,
\[
\nu^{-1}\Qp_{3/2,I}=2A+B
\le2A+\frac{18}{16}B
=2\iint|\nabla Z|^2.
\]
Multiplying by $\nu$ proves both inequalities with the stated constants.
\end{proof}

\medskip
\noindent\textbf{Parabolic scaling.}
\begin{proposition}\label{prop:scaling}
Under $u_r(x,t)=r\,u(x_0+rx,t_0+r^2t)$,
\[
\boxed{\Qp_{p}[u_r]=r^{2p-3}\Qp_p[u;Q_r].}
\]
Hence
\[
\boxed{p_c=\frac32.}
\]
\end{proposition}
\begin{proof}
To avoid confusing the scaling parameter with the rescaled variables, write
\[
X=x_0+rx,
\qquad
T=t_0+r^2t.
\]
Then
\[
u_r(x,t)=r\,u(X,T).
\]
A spatial derivative with respect to $x$ contributes one additional factor
$r$, so
\[
\nabla_xu_r=r^2(\nabla_Xu)(X,T).
\]
Taking the curl gives
\[
\omega_r(x,t)=r^2\omega(X,T).
\]
Hence
\[
f_r=r^2f(X,T),
\qquad
\xi_r=\xi(X,T)
\]
where $f>0$.  Differentiating these quantities once more in $x$ gives
\[
\nabla_x\xi_r=r(\nabla_X\xi)(X,T),
\qquad
\nabla_xf_r=r^3(\nabla_Xf)(X,T).
\]
The spacetime Jacobian is
\[
\dd x\,\dd t=r^{-3}r^{-2}\,\dd X\,\dd T=r^{-5}\,\dd X\,\dd T.
\]

Apply the dissipative representation in \cref{thm:layercake}.  The
directional integrand scales as
\[
f_r^p|\nabla\xi_r|^2
=r^{2p}r^2 f^p|\nabla\xi|^2,
\]
and after multiplication by the Jacobian its total factor is
$r^{2p+2-5}=r^{2p-3}$.  The amplitude integrand scales as
\[
f_r^{p-2}|\nabla f_r|^2
=r^{2(p-2)}r^6 f^{p-2}|\nabla f|^2,
\]
which has the same total factor
\[
r^{2p-4+6-5}=r^{2p-3}.
\]
Therefore the complete quantity satisfies
\[
\Qp_p[u_r]=r^{2p-3}\Qp_p[u;Q_r].
\]
Scale invariance occurs exactly when the exponent vanishes:
\[
2p-3=0\quad\Longleftrightarrow\quad p=\frac32.
\]
This proves both assertions.
\end{proof}

\section{Twist versus transverse amplitude sheath}
The next elementary estimate explains how the divergence-free constraint
couples variation of the vorticity magnitude along a vortex line to variation
of the vorticity direction.

\medskip
\noindent\textbf{Longitudinal amplitude is controlled by directional twist.}
\begin{proposition}\label{prop:twist-sheath}
On $\{f>0\}$,
\[
\boxed{\xi\cdot\nabla f=-f\,\nabla\cdot\xi.}
\]
Moreover,
\[
|\nabla\cdot\xi|^2\le2|\nabla\xi|^2.
\]
Consequently, if
\[
\nabla_\parallel f=(\xi\cdot\nabla f)\xi,
\qquad
\nabla_\perp f=(I-\xi\otimes\xi)\nabla f,
\]
and
\[
\mathcal T=\iint_I f^{3/2}|\nabla\xi|^2,
\qquad
\mathcal S_\perp=\iint_I f^{-1/2}|\nabla_\perp f|^2,
\]
then
\[
\boxed{
2\mathcal T+\mathcal S_\perp
\le \nu^{-1}\Qp_{3/2,I}
\le4\mathcal T+\mathcal S_\perp.
}
\]
\end{proposition}
\begin{proof}
Since $\omega=f\xi$ and $\nabla\cdot\omega=0$, the product rule gives
\[
0=\nabla\cdot(f\xi)=\xi\cdot\nabla f+f\,\nabla\cdot\xi,
\]
which is the first identity.

The inequality for $\nabla\cdot\xi$ is pointwise and rotationally invariant,
so at a fixed point we may choose orthonormal coordinates with
$\xi=e_3$.  Differentiating $|\xi|^2=1$ gives
\[
\xi\cdot\partial_j\xi=0
\qquad (j=1,2,3).
\]
At the chosen point this means $\partial_j\xi_3=0$ for every $j$.  Hence
\[
\nabla\cdot\xi
=\partial_1\xi_1+\partial_2\xi_2+\partial_3\xi_3
=\partial_1\xi_1+\partial_2\xi_2.
\]
By $(a+b)^2\le2(a^2+b^2)$,
\[
|\nabla\cdot\xi|^2
\le2\bigl(|\partial_1\xi_1|^2+|\partial_2\xi_2|^2\bigr)
\le2|\nabla\xi|^2.
\]
Combining this with the first identity yields
\[
|\nabla_\parallel f|^2
=|\xi\cdot\nabla f|^2
=f^2|\nabla\cdot\xi|^2
\le2f^2|\nabla\xi|^2.
\]
Therefore
\[
\iint f^{-1/2}|\nabla_\parallel f|^2
\le2\iint f^{3/2}|\nabla\xi|^2
=2\mathcal T.
\]

Now decompose $\nabla f$ orthogonally:
\[
|\nabla f|^2
=|\nabla_\parallel f|^2+|\nabla_\perp f|^2.
\]
At the critical exponent, \cref{thm:layercake} reads
\[
\nu^{-1}\Qp_{3/2,I}
=2\mathcal T
+\iint f^{-1/2}|\nabla_\parallel f|^2
+\mathcal S_\perp.
\]
The middle term is nonnegative, which gives the lower bound
\[
2\mathcal T+\mathcal S_\perp
\le\nu^{-1}\Qp_{3/2,I}.
\]
Using its upper estimate by $2\mathcal T$ gives
\[
\nu^{-1}\Qp_{3/2,I}
\le4\mathcal T+\mathcal S_\perp.
\]
This proves the proposition.  The estimate is an exact geometric
decomposition of the critical dissipation; no regularity conclusion is being
used or asserted here.
\end{proof}

\medskip
\noindent\textbf{Zero-twist integrable vorticity.}
\begin{lemma}\label{lem:zerotwist}
Let $\Omega=f\xi\in C(\R^3)\cap L^q(\R^3)$, $1\le q<\infty$, be divergence free. If $\nabla\xi=0$ on each connected component of $\{f>0\}$, then $\Omega\equiv0$.
\end{lemma}
\begin{proof}
Let $\mathcal O$ be a connected component of the open set $\{f>0\}$.  By
hypothesis $\nabla\xi=0$ on $\mathcal O$, so $\xi$ is constant there; write
$\xi=e$ with $|e|=1$.  Since $\Omega=fe$ is divergence free,
\[
0=\nabla\cdot(fe)=e\cdot\nabla f
\qquad\text{on }\mathcal O.
\]
Thus $f$ is constant along every line segment parallel to $e$ that remains
inside $\mathcal O$.

Suppose, toward a contradiction, that $\mathcal O$ is nonempty.  Choose
$x_0\in\mathcal O$ and set $c_0=f(x_0)>0$.  By continuity and openness there
is $r_0>0$ such that
\[
B_{r_0}(x_0)\subset\mathcal O,
\qquad
f\ge\frac{c_0}{2}
\quad\text{on }B_{r_0}(x_0).
\]
Let $D$ be the two-dimensional disk of radius $r_0/2$ in the plane through
$x_0$ perpendicular to $e$.  For each $y\in D$, let $I_y$ be the maximal
open interval containing $0$ for which
\[
y+se\in\mathcal O\qquad(s\in I_y).
\]
Because $y$ lies in the smaller disk, $I_y$ contains a nontrivial interval
around $0$, and $f(y)\ge c_0/2$.  The identity $e\cdot\nabla f=0$ implies
\[
f(y+se)=f(y)\ge\frac{c_0}{2}
\qquad(s\in I_y).
\]
We claim that $I_y=\R$.  If, for example, the right endpoint $s_+<\infty$,
then $y+s_+e$ belongs to the boundary of $\mathcal O$.  Since
$\mathcal O$ is a component of $\{f>0\}$ and $f$ is continuous, one must have
$f(y+s_+e)=0$.  But continuity along the line gives
\[
f(y+s_+e)=\lim_{s\uparrow s_+}f(y+se)=f(y)\ge c_0/2,
\]
a contradiction.  The left endpoint is treated identically.  Therefore the
whole line $y+\R e$ lies in $\mathcal O$ and carries the lower bound
$f\ge c_0/2$.

As $y$ ranges over the disk $D$, these lines form an infinite cylinder of
positive cross-sectional area.  Hence
\[
\int_{\R^3}|\Omega|^q\,\dd x
=\int_{\R^3}f^q\,\dd x
\ge |D|\int_{\R}\left(\frac{c_0}{2}\right)^q\dd s
=\infty,
\]
contradicting $\Omega\in L^q(\R^3)$.  Thus no nonempty component
$\mathcal O$ can exist, so $f=0$ everywhere and $\Omega\equiv0$.
\end{proof}

\section{Threshold-dissipation measure and arbitrary-overlap first assignment}
Let
\[
\mathcal X_I=I\times\R^3\times(0,\infty).
\]
For $p>1$ define the nonnegative Radon measure $\mathfrak m_p$ on
$\mathcal X_I$ by
\[
\begin{aligned}
\mathfrak m_p(B)={}&
\nu\int_I\int_{\R^3}\int_0^{f(x,t)}
\mathbf 1_B(t,x,a)\,a^{p-2}f|\nabla\xi|^2\,\dd a\dd x\dd t\\
&+\nu\int_I\int_{\R^3}
\mathbf 1_B(t,x,f(x,t))\,f^{p-2}|\nabla f|^2\,\dd x\dd t.
\end{aligned}
\]
We interpret the second integrand as zero on $\{f=0\}$; this convention is
irrelevant because the threshold coordinate of $\mathcal X_I$ is strictly
positive.

\medskip
\noindent\textbf{Radon property and total mass.}
\begin{proposition}\label{prop:threshold-measure}
For every $p>1$, the preceding formula defines a nonnegative Radon measure on
$\mathcal X_I$.  Its total mass, possibly $+\infty$, is
\[
\boxed{
\mathfrak m_p(\mathcal X_I)=\Qp_{p,I}.
}
\]
At $p=2$, its spacetime marginal is exactly
\[
\boxed{
\nu|\nabla\omega|^2\,\dd x\,\dd t,
}
\]
and therefore
\[
\boxed{\mathfrak m_2(\mathcal X_I)=\nu\iint_I|\nabla\omega|^2.}
\]
\end{proposition}
\begin{proof}
Separate the definition into two measures,
$\mathfrak m_p=\mathfrak m_p^{\rm dir}+\mathfrak m_p^{\rm amp}$.
The directional part is absolutely continuous with respect to
$\dd t\,\dd x\,\dd a$ on the subgraph $0<a<f(x,t)$, with nonnegative density
\[
\nu a^{p-2}f|\nabla\xi|^2.
\]
Let $K\subset\mathcal X_I$ be compact.  Since the threshold coordinate is in
$(0,\infty)$, there exist $0<a_-<a_+<\infty$ such that every
$(t,x,a)\in K$ has $a_-\le a\le a_+$.  The spatial projection of $K$ is also
compact.  On this compact spacetime-threshold region the smooth solution and
its derivatives are bounded, while $a^{p-2}$ is bounded because
$a\ge a_->0$.  Thus $\mathfrak m_p^{\rm dir}(K)<\infty$.

For the amplitude part, define the map
\[
\Psi:I\times\{x:f(x,t)>0\}\longrightarrow\mathcal X_I,
\qquad
\Psi(t,x)=(t,x,f(x,t)).
\]
Then
\[
\mathfrak m_p^{\rm amp}
=\Psi_\#\left(
\nu\mathbf1_{\{f>0\}}f^{p-2}|\nabla f|^2\,\dd x\,\dd t
\right).
\]
If $K$ is compact as above, its preimage under $\Psi$ lies where
$f\ge a_->0$ and in a compact spacetime set.  Hence the density is bounded
there and the preimage has finite measure.  Thus
$\mathfrak m_p^{\rm amp}(K)<\infty$.  Both pieces are Borel and locally finite
on the locally compact metric space $\mathcal X_I$, hence are Radon measures.

To compute the total mass of the directional part, Tonelli gives
\[
\begin{aligned}
\mathfrak m_p^{\rm dir}(\mathcal X_I)
&=\nu\int_I\int_{\R^3}
 f|\nabla\xi|^2\left(\int_0^f a^{p-2}\,\dd a\right)\dd x\,\dd t\\
&=\frac{\nu}{p-1}\iint_I f^p|\nabla\xi|^2.
\end{aligned}
\]
The amplitude part has total mass
\[
\mathfrak m_p^{\rm amp}(\mathcal X_I)
=\nu\iint_I f^{p-2}|\nabla f|^2.
\]
Their sum is exactly the expression for $\Qp_{p,I}$ in
\cref{thm:layercake}.

At $p=2$, project $\mathfrak m_2$ onto spacetime.  Integrating the directional
part over $0<a<f$ produces the density
\[
\nu f^2|\nabla\xi|^2,
\]
while the amplitude part projects to
\[
\nu|\nabla f|^2.
\]
By the orthogonal decomposition used in \cref{cor:pal},
\[
|\nabla\omega|^2=|\nabla f|^2+f^2|\nabla\xi|^2.
\]
Therefore the spacetime marginal is
$\nu|\nabla\omega|^2\,\dd x\,\dd t$, and integrating it gives the final
identity.
\end{proof}

\medskip
\noindent\textbf{Exact first assignment for arbitrary overlap.}
\begin{theorem}\label{thm:first-assigned}
Let $E_1,E_2,\dots\subset\mathcal X_I$ be any ordered Borel family and define
\[
F_j=E_j\setminus\bigcup_{i<j}E_i.
\]
Then the $F_j$ are pairwise disjoint and
\[
\boxed{
\sum_j\mathfrak m_p(F_j)
=\mathfrak m_p\!\left(\bigcup_jE_j\right)
\le\mathfrak m_p(\mathcal X_I).
}
\]
If moreover $\mathfrak m_p(E_j)\ge q_*>0$ for all $j$ and
$M=\mathfrak m_p(\mathcal X_I)<\infty$, then for every $\eta>0$,
\[
\#\{j:\mathfrak m_p(F_j)\ge\eta q_*\}
\le \frac{M}{\eta q_*},
\]
and therefore
\[
\boxed{
\frac{\mathfrak m_p(F_j)}{\mathfrak m_p(E_j)}\to0.
}
\]
\end{theorem}
\begin{proof}
We begin with the set theory.  If $i<j$ and a point belonged to both $F_i$ and
$F_j$, then it would belong to $E_i$; but the definition
\[
F_j=E_j\setminus\bigcup_{k<j}E_k
\]
removes every point of $E_i$ from $F_j$.  Hence the $F_j$ are pairwise
disjoint.

Next we verify that disjointification does not change the union.  Clearly
$F_j\subset E_j$, so
\[
\bigcup_jF_j\subset\bigcup_jE_j.
\]
Conversely, if $z\in\bigcup_jE_j$, let $j_0$ be the smallest index for which
$z\in E_{j_0}$.  Then $z\notin E_i$ for $i<j_0$, hence
$z\in F_{j_0}$.  Therefore
\[
\bigcup_jF_j=\bigcup_jE_j.
\]
Since $\mathfrak m_p$ is a positive countably additive measure and the
$F_j$ are disjoint,
\[
\sum_{j=1}^\infty\mathfrak m_p(F_j)
=\mathfrak m_p\!\left(\bigcup_jF_j\right)
=\mathfrak m_p\!\left(\bigcup_jE_j\right)
\le\mathfrak m_p(\mathcal X_I).
\]
This proves the first assertion.

Now assume $M=\mathfrak m_p(\mathcal X_I)<\infty$ and
$\mathfrak m_p(E_j)\ge q_*>0$.  Fix $\eta>0$ and let
\[
J_\eta=\{j:\mathfrak m_p(F_j)\ge\eta q_*\}.
\]
Because the $F_j$ are disjoint,
\[
M\ge\sum_{j\in J_\eta}\mathfrak m_p(F_j)
\ge\#J_\eta\,\eta q_*.
\]
Thus
\[
\#J_\eta\le\frac{M}{\eta q_*}.
\]
In particular $J_\eta$ is finite.  Therefore, for all sufficiently large
$j$,
\[
\mathfrak m_p(F_j)<\eta q_*\le\eta\mathfrak m_p(E_j),
\]
so
\[
0\le\frac{\mathfrak m_p(F_j)}{\mathfrak m_p(E_j)}<\eta.
\]
Since $\eta>0$ was arbitrary, the ratio tends to zero.
\end{proof}

\medskip
\noindent\textbf{Persistent-reuse core.}
\begin{corollary}\label{cor:reuse-core}
Assume $\mathfrak m_p(\mathcal X_I)<\infty$ and
$\mathfrak m_p(E_j)\ge q_*>0$ for every $j$.  Then
\[
\mathcal R:=\limsup_{j\to\infty}E_j
\quad\text{satisfies}\quad
\boxed{\mathfrak m_p(\mathcal R)\ge q_*.}
\]
\end{corollary}
\begin{proof}
For $N\ge1$ set
\[
A_N=\bigcup_{j\ge N}E_j.
\]
The sets $A_N$ form a decreasing sequence:
$A_{N+1}\subset A_N$.  Moreover $E_N\subset A_N$, so
\[
\mathfrak m_p(A_N)\ge\mathfrak m_p(E_N)\ge q_*
\qquad\text{for every }N.
\]
By definition,
\[
\bigcap_{N=1}^\infty A_N
=\bigcap_{N=1}^\infty\bigcup_{j\ge N}E_j
=\limsup_{j\to\infty}E_j
=\mathcal R.
\]
The finite-total-mass assumption implies
$\mathfrak m_p(A_1)\le\mathfrak m_p(\mathcal X_I)<\infty$, so continuity of a
measure from above applies:
\[
\mathfrak m_p(\mathcal R)
=\lim_{N\to\infty}\mathfrak m_p(A_N).
\]
Every term on the right is at least $q_*$, hence the limit is at least $q_*$.
\end{proof}

At $p=2$, \cref{prop:threshold-measure} identifies the spacetime marginal
with $\nu|\nabla\omega|^2\,\dd x\dd t$.  If $(x_*,t_*)$ is a smooth interior
point, then $|\nabla\omega|$ is bounded on some spacetime neighborhood of that
point.  Hence the marginal mass of a parabolic cylinder of radius $r$ centered
there is bounded by a constant times its spacetime volume, and therefore tends
to zero as $r\downarrow0$.  In particular the marginal has no atom at a smooth
interior point.  Thus a nested family carrying a uniform positive amount of
this measure cannot collapse silently to such a point.  Any such collapse must
instead approach the terminal boundary, lose local Radon boundedness, escape in
amplitude/space/time, or create one of the explicit cutoff, tail, or transition
terms of the localized identity.

\begin{remark}
The time-disjoint packing theorem is the special case in which the
first-assigned sets are already disjoint in time.  The phase-space ownership
above is strictly stronger and does not require a bounded raw overlap
constant.
\end{remark}

\section{Exact renormalization covariance and stationary compactification}
Let selected packets be $Q_j=(x_j,t_j,r_j)$ and set
\[
\begin{aligned}
U_j(y,s)&=r_j u(x_j+r_jy,t_j+r_j^2s),\\
P_j(y,s)&=r_j^2p(x_j+r_jy,t_j+r_j^2s),\\
\Omega_j(y,s)&=r_j^2\omega(x_j+r_jy,t_j+r_j^2s).
\end{aligned}
\]
With
\[
\lambda_j=\frac{r_{j+1}}{r_j},\qquad
 a_j=\frac{x_{j+1}-x_j}{r_j},\qquad
 b_j=\frac{t_{j+1}-t_j}{r_j^2},
\]
the change from one normalized packet to the next is completely explicit.

\medskip
\noindent\textbf{Exact covariance of fields and threshold measure.}
\begin{proposition}\label{prop:renorm-covariance}
The normalized velocities satisfy
\[
\boxed{
U_{j+1}(Y,S)=\lambda_jU_j(a_j+\lambda_jY,b_j+\lambda_j^2S).
}
\]
The corresponding pressure and vorticity satisfy the same relation with
prefactor $\lambda_j^2$, with pressure understood modulo addition of a
function of time.  If
\[
\Phi_j(x,t,a)=\left(\frac{x-x_j}{r_j},\frac{t-t_j}{r_j^2},r_j^2a\right),
\qquad
\mu_j=r_j^{2p-3}(\Phi_j)_\#\mathfrak m_p,
\]
and
\[
\mathcal A_j(y,s,\beta)
=\left(\frac{y-a_j}{\lambda_j},\frac{s-b_j}{\lambda_j^2},\lambda_j^2\beta\right),
\]
then
\[
\boxed{
\mu_{j+1}=\lambda_j^{2p-3}(\mathcal A_j)_\#\mu_j.
}
\]
\end{proposition}
\begin{proof}
Start from the definition of $U_{j+1}$:
\[
U_{j+1}(Y,S)
=r_{j+1}u(x_{j+1}+r_{j+1}Y,t_{j+1}+r_{j+1}^2S).
\]
Use
\[
r_{j+1}=\lambda_jr_j,
\qquad
x_{j+1}=x_j+r_ja_j,
\qquad
t_{j+1}=t_j+r_j^2b_j.
\]
Then
\[
\begin{aligned}
x_{j+1}+r_{j+1}Y
&=x_j+r_j(a_j+\lambda_jY),\\
t_{j+1}+r_{j+1}^2S
&=t_j+r_j^2(b_j+\lambda_j^2S).
\end{aligned}
\]
Substitution gives
\[
\begin{aligned}
U_{j+1}(Y,S)
&=\lambda_jr_j
u\!\left(x_j+r_j(a_j+\lambda_jY),
 t_j+r_j^2(b_j+\lambda_j^2S)\right)\\
&=\lambda_jU_j(a_j+\lambda_jY,b_j+\lambda_j^2S).
\end{aligned}
\]
The pressure formula is identical except that its normalization is $r_j^2$,
so the prefactor is $\lambda_j^2$.  The pressure class $[P_j]$ is used because
Navier--Stokes determines pressure only up to a function of time.  The
vorticity also has scaling degree two, so it receives the same prefactor.

For the measure, first observe the exact composition law for the coordinate
maps.  If
\[
(y,s,\beta)=\Phi_j(x,t,a),
\]
then
\[
y=\frac{x-x_j}{r_j},\qquad
s=\frac{t-t_j}{r_j^2},\qquad
\beta=r_j^2a.
\]
Therefore
\[
\begin{aligned}
\Phi_{j+1}(x,t,a)
&=\left(
\frac{x-x_{j+1}}{r_{j+1}},
\frac{t-t_{j+1}}{r_{j+1}^2},
r_{j+1}^2a
\right)\\
&=\left(
\frac{y-a_j}{\lambda_j},
\frac{s-b_j}{\lambda_j^2},
\lambda_j^2\beta
\right)\\
&=\mathcal A_j(\Phi_j(x,t,a)).
\end{aligned}
\]
Hence
\[
(\Phi_{j+1})_\#\mathfrak m_p
=(\mathcal A_j)_\#(\Phi_j)_\#\mathfrak m_p.
\]
Multiplying by the normalization factor and using
$r_{j+1}^{2p-3}=\lambda_j^{2p-3}r_j^{2p-3}$ gives
\[
\begin{aligned}
\mu_{j+1}
&=r_{j+1}^{2p-3}(\Phi_{j+1})_\#\mathfrak m_p\\
&=\lambda_j^{2p-3}(\mathcal A_j)_\#
\left(r_j^{2p-3}(\Phi_j)_\#\mathfrak m_p\right)\\
&=\lambda_j^{2p-3}(\mathcal A_j)_\#\mu_j.
\end{aligned}
\]
This is an exact coordinate identity.  Thus the canonical rescaling itself
creates no commutator.  Any transition error must come from an additional
operation---for example truncating a tail, changing the quotient, or replacing
the canonical map by an approximation.
\end{proof}

\medskip
\noindent\textbf{Stationary Markov law from a closed transition relation.}
\begin{theorem}\label{thm:stationary-law}
Assume the decorated states
\[
\mathbf S_j=(U_j,[P_j],\mu_j)
\]
lie in one compact metric state chamber $\mathcal K$, with the measure
coordinate carrying the weak-* topology on Radon measures.  Let
$\mathcal T\subset\mathcal K\times\mathcal K$ be a closed relation such that
\[
(\mathbf S_j,\mathbf S_{j+1})\in\mathcal T
\qquad\text{for every }j.
\]
Assume every state retains a fixed threshold mass
$\mu_j(K_*)\ge q_*>0$ on one closed compact reference chamber $K_*$.  Then
there exist a probability measure $\rho$ on $\mathcal K$ and a Markov kernel
$P$ on $\mathcal K$ such that
\[
\boxed{\rho P=\rho,}
\]
the transition coupling $\rho(\dd S)P(S,\dd S')$ is supported on
$\mathcal T$, and $\rho$ is supported on active states with threshold mass at
least $q_*$.  No single-valued or continuous return map is assumed.
\end{theorem}
\begin{proof}
For each $N$ define the empirical transition measure
\[
\Pi_N=\frac1N\sum_{j=1}^N
\delta_{(\mathbf S_j,\mathbf S_{j+1})}
\quad\text{on }\mathcal K\times\mathcal K.
\]
The product is compact metric, hence its probability measures are weakly
sequentially compact.  Along a subsequence,
\[
\Pi_N\rightharpoonup\Pi.
\]
Every $\Pi_N$ gives full mass to the closed set $\mathcal T$.  Portmanteau
therefore gives $\Pi(\mathcal T)=1$.

Let $\rho_N^{(1)}$ and $\rho_N^{(2)}$ be the two state marginals of $\Pi_N$.
Then
\[
\rho_N^{(1)}-\rho_N^{(2)}
=\frac1N\bigl(\delta_{\mathbf S_1}-\delta_{\mathbf S_{N+1}}\bigr),
\]
so their total-variation distance is at most $2/N$.  Marginalization is
continuous under weak convergence on a compact product, hence the two
marginals of $\Pi$ coincide.  Denote the common marginal by $\rho$.

Since $\mathcal K$ is compact metric, it is a standard Borel space.
Disintegrating $\Pi$ with respect to its first marginal gives a Markov kernel
$P$ such that
\[
\Pi(\dd S,\dd S')=\rho(\dd S)P(S,\dd S').
\]
The second marginal is also $\rho$, and therefore, for every Borel
$A\subset\mathcal K$,
\[
(\rho P)(A)
=\Pi(\mathcal K\times A)
=\rho(A).
\]
Thus $\rho P=\rho$.  Since $\Pi(\mathcal T)=1$, the stationary transition
coupling is supported on the exact closed relation.

It remains to retain activity.  Put
\[
\mathcal C_*
=\{S=(U,[P],\mu)\in\mathcal K:\mu(K_*)\ge q_*\}.
\]
This set is closed: if $\mu_n\rightharpoonup^*\mu$, then for closed $K_*$,
Portmanteau gives
\[
\limsup_n\mu_n(K_*)\le\mu(K_*).
\]
Thus $\mu_n(K_*)\ge q_*$ for all $n$ implies $\mu(K_*)\ge q_*$.  Every
empirical first marginal is supported on $\mathcal C_*$, so another
Portmanteau argument gives $\rho(\mathcal C_*)=1$.
\end{proof}

\begin{remark}\label{rem:atomless-stationary}
The theorem does not imply that $\rho$ has an atom or that one normalized
state recurs exactly.  A compact dynamical system may preserve an atomless
probability measure; irrational rotation of the circle with Lebesgue measure
is the elementary example.  Any later argument requiring a recurrent atom,
a deterministic profile, or an exact return must prove that additional
property separately.
\end{remark}

\begin{remark}\label{rem:transition-parameters}
If scale, center, time-shift, frame, or other transition parameters must be
retained, they may be added as marks to the transition relation (or to the
state).  The empirical-edge proof is unchanged.  The essential hypotheses are
compactness of the decorated state sequence and closedness of the physical
same-solution transition relation, not continuity of a globally defined
return map.
\end{remark}

\medskip
\noindent\textbf{Temporal concentration atom.}
\begin{proposition}\label{prop:time-atom}
Let nonnegative finite Radon measures $\eta_n\rightharpoonup^*\eta$ on a compact
time interval.  If closed intervals $J_n$ shrink to $\{s_*\}$ and
$\eta_n(J_n)\ge q_*>0$, then
\[
\boxed{\eta(\{s_*\})\ge q_*.}
\]
\end{proposition}
\begin{proof}
Let the ambient compact time interval be $K$.  Fix $\delta>0$ and define the
closed neighborhood
\[
F_\delta=K\cap[s_*-\delta,s_*+\delta].
\]
The assumption that $J_n$ shrinks to $\{s_*\}$ means that for every
$\delta>0$ there is $n_\delta$ such that
\[
J_n\subset F_\delta
\qquad(n\ge n_\delta).
\]
Hence
\[
\eta_n(F_\delta)\ge\eta_n(J_n)\ge q_*
\qquad(n\ge n_\delta).
\]
Because $F_\delta$ is closed and $\eta_n\rightharpoonup^*\eta$, the
Portmanteau theorem gives
\[
\limsup_{n\to\infty}\eta_n(F_\delta)\le\eta(F_\delta).
\]
Therefore
\[
\eta(F_\delta)\ge q_*
\qquad\text{for every }\delta>0.
\]
As $\delta\downarrow0$, the sets $F_\delta$ decrease to $\{s_*\}$.  Since
$\eta$ is finite on the compact interval $K$, continuity from above applies:
\[
\eta(\{s_*\})
=\lim_{\delta\downarrow0}\eta(F_\delta)
\ge q_*.
\]
Thus the limiting measure has an atom of at least the stated mass.
\end{proof}

The preceding theorem should be read as a compactification principle, not as a
claim that compactness follows from the threshold identity alone.  In an
endpoint argument, failure to place the normalized sequence in a compact
decorated chamber, or failure of closedness of the physical transition
relation, must be diagnosed separately.  If that external compactness analysis
has already reduced every failure to local energy/pressure/threshold-measure
blow-up or to scale/space/time/amplitude escape, then
\cref{thm:stationary-law,prop:time-atom} give the conditional alternative
\[
\boxed{
\begin{aligned}
\text{infinite fixed-margin sequence}\Longrightarrow{}&
\text{local energy/pressure/threshold-measure blow-up}\\
&\vee\text{scale/space/time/amplitude escape}\\
&\vee\text{temporal concentration atom}\\
&\vee\text{stationary active Markov renormalization law}.
\end{aligned}}
\]
The first two lines summarize failures of the compact-chamber hypotheses; the
last two are the conclusions available once compactness is retained.  This
principle does not assert axisymmetry, periodicity, deterministic recurrence,
smooth profile realization, or regularity.

\section{Limitations and endpoint interpretation}
The threshold-dissipation measure is attached to the amplitude-level identity.  It does not contain the full pressure, local-energy, nonlinear-product, tail, time-face, or signed flux data of a suitable-weak endpoint.  Taking positive parts of signed localized terms would change the PDE identity and is not used.

The first-assignment theorem solves only the overlap-accounting problem for the positive threshold measure.  A positive-mass persistent set need not determine a unique nested lineage, and the stationary Markov law obtained under compactness hypotheses need not reduce to a deterministic self-similar profile or even possess an atom.  Subsequent endpoint analysis must therefore retain the actual rescaling relation and the complete localized PDE state.

Accordingly, the results of this paper provide a threshold representation, a nonoverlapping allocation, and a compactification principle.  They do not constitute a first-singularity classification and do not imply unconditional regularity.

\appendix
\section{Classical vorticity balance}
For completeness we record the classical balance in a form that can be
checked directly against the threshold reconstruction.  No additional
regularity conclusion is used here.

\begin{proposition}\label{prop:appendix-lp}
Let $p>1$.  Under the smoothness and decay assumptions of the main text,
\[
\begin{aligned}
\frac1p\left[\int_{\R^3}|\omega|^p\,\dd x\right]_{t_0}^{t_1}
&+\nu(p-1)\iint_I |\omega|^{p-2}|\nabla|\omega||^2\,\dd x\,\dd t\\
&+\nu\iint_I |\omega|^p|\nabla\xi|^2\,\dd x\,\dd t
=\iint_I \alpha|\omega|^p\,\dd x\,\dd t.
\end{aligned}
\]
At $p=2$ this becomes the enstrophy--palinstrophy identity
\[
\boxed{
\frac12\left[\int|\omega|^2\right]_{t_0}^{t_1}
+\nu\iint_I|\nabla\omega|^2
=\iint_I\alpha|\omega|^2.
}
\]
\end{proposition}
\begin{proof}
The first formula is exactly \cref{cor:classical-lp}; we include a direct
calculation as a consistency check.  On $\{f>0\}$,
\cref{lem:magnitude} reads
\[
(\partial_t+u\cdot\nabla)f
=\nu\Delta f+\alpha f-\nu f|\nabla\xi|^2.
\]
Formally multiply by $f^{p-1}$.  If $1<p<2$, one may justify the calculation
by replacing $f^{p-1}$ with
$(f^2+\varepsilon^2)^{(p-2)/2}f$ and then sending
$\varepsilon\downarrow0$; for a smooth decaying solution this gives the same
limit.  The time term is
\[
f^{p-1}\partial_tf=\frac1p\partial_t(f^p).
\]
The transport term is also a perfect derivative:
\[
f^{p-1}u\cdot\nabla f
=\frac1p u\cdot\nabla(f^p),
\]
whose spatial integral vanishes because $\nabla\cdot u=0$ and the fields
decay.  For diffusion,
\[
\begin{aligned}
\int f^{p-1}\Delta f\,\dd x
&=-\int\nabla(f^{p-1})\cdot\nabla f\,\dd x\\
&=-(p-1)\int f^{p-2}|\nabla f|^2\,\dd x.
\end{aligned}
\]
The remaining two terms become
\[
\int \alpha f^p\,\dd x
\qquad\text{and}\qquad
-\nu\int f^p|\nabla\xi|^2\,\dd x.
\]
Therefore
\[
\frac1p\frac{\dd}{\dd t}\int f^p
+\nu(p-1)\int f^{p-2}|\nabla f|^2
+\nu\int f^p|\nabla\xi|^2
=\int\alpha f^p.
\]
Integrating in time yields the first displayed formula.

When $p=2$, the orthogonal polar decomposition
\[
|\nabla\omega|^2=|\nabla f|^2+f^2|\nabla\xi|^2
\]
combines the two viscous terms into $\nu\int|\nabla\omega|^2$, proving the
second formula.  This direct derivation agrees exactly with the result obtained
by integrating the threshold identity over $a$.
\end{proof}

\end{document}